\documentclass[letterpaper, 10 pt, conference]{ieeeconf}  % Comment this line out if you need a4paper

\IEEEoverridecommandlockouts                              % This command is only needed if 
\input{irom.sty}
\newtheorem{theorem}{Theorem}[section]

\usepackage{etoolbox}
\usepackage{mathtools}
\usepackage{eucal}
\usepackage{ftnright}
\newtoggle{extended}
\toggletrue{extended} % Use this for extended version (will include proofs, etc.)
\usepackage[capitalise]{cleveref}
\crefformat{equation}{(#2#1#3)}
\crefrangeformat{equation}{(#3#1#4) to~(#5#2#6)}

\newcommand{\DP}{\mathcal{DP}}

\title{\LARGE \bf
Task-Driven Estimation and Control via Information Bottlenecks
}

\author{Vincent Pacelli and Anirudha Majumdar% <-this % stops a space
\thanks{This work was supported by the National Science Foundation (NSF) [IIS-1755038] and a Google Faculty Research Award.}% <-this % stops a space
\thanks{The authors are with the Mechanical and Aerospace Engineering department at Princeton University, NJ, 08540, USA
        {\tt\small \{vpacelli, ani.majumdar\}@princeton.edu}%
}%
}

\iftoggle{extended}{\def\baselinestretch{0.98}}{\def\baselinestretch{0.98}}

\begin{document}

\maketitle
\thispagestyle{empty}
\pagestyle{empty}

%%%%%%%%%%%%%%%%%%%%%%%%%%%%%%%%%%%%%%%%%%%%%%%%%%%%%%%%%%%%%%%%%%%%%%%%%%%%%%%%
\begin{abstract}
Our goal is to develop a principled and general algorithmic framework for \emph{task-driven estimation and control} for robotic systems. State-of-the-art approaches for controlling robotic systems typically rely heavily on accurately estimating the full state of the robot (e.g., a running robot might estimate joint angles and velocities, torso state, and position relative to a goal). However, full state representations are often excessively rich for the specific task at hand and can lead to significant computational inefficiency and brittleness to errors in state estimation. In contrast, we present an approach that eschews such rich representations and seeks to create \emph{task-driven representations}. The key technical insight is to leverage the theory of \emph{information bottlenecks} to formalize the notion of a ``task-driven representation" in terms of information theoretic quantities that measure the \emph{minimality} of a representation. We propose novel iterative algorithms for automatically synthesizing (offline) a task-driven representation (given in terms of a set of task-relevant variables (TRVs)) and a performant control policy that is a function of the TRVs. We present \emph{online} algorithms for estimating the TRVs in order to apply the control policy. We demonstrate that our approach results in significant robustness to \emph{unmodeled} measurement uncertainty both theoretically and via thorough simulation experiments including a spring-loaded inverted pendulum running to a goal location.

%Extensive experimental studies have shown that when performing dexterous activities like catching a ball or running, humans do not maintain a very detailed representation of the full state of the environment. Instead, humans choose actions based on a few task-relevant variables (TRVs) that are sufficient for achieving the desired goal (a classical example is the so-called "gaze heuristic" for catching a ball). This behavior is believed to result in a significantly reduced cognitive load and increased accuracy due to the removal of unnecessary sources of uncertainty. Motivated by these observations, the goal of our work is to develop a principled and general framework for task-driven perception and control for robotic systems. In particular, we seek to algorithmically identify and exploit TRVs to create computationally efficient and robust controllers for robotic tasks. To this end, we introduce an efficient iterative algorithm that produces both a performant control policy and set of low-information TRVs that are sufficient for control. The computational and robustness benefits of this approach are quantified using an information-theoretic framework. Finally, the efficacy of this algorithm is demonstrated through comparison with both traditional robotics methods and canonical cognitive psychology experiments on a variety of examples.
\end{abstract}

%%%%%%%%%%%%%%%%%%%%%%%%%%%%%%%%%%%%%%%%%%%%%%%%%%%%%%%%%%%%%%%%%%%%%%%%%%%%%%%%
\section{Introduction}

State-of-the-art techniques for controlling robotic systems typically rely heavily on accurately estimating the full state of the system and maintaining rich geometric representations of their environment. For example, a common approach to navigation is to build a dense occupancy map produced by scanning the environment and to use this map for planning and control. Similarly, control for walking or running robots typically involves estimating the full state of the robot (e.g., joint angles, velocities, and position relative to a goal location). However, such representations are often overly detailed when compared to a \emph{task-driven representation}.

One example of a task-driven representation is the ``gaze heuristic'' from cognitive psychology \cite{Gigerenzer07, McLeod03, Shaffer04}. When attempting to catch a ball, an agent can estimate the ball's position and velocity, model how it will evolve in conjunction with environmental factors like wind, integrate the pertinent differential equations, and plan a trajectory in order to arrive at the ball's final location. In contrast, cognitive psychology studies have shown that humans use a dramatically simpler strategy that entails maintaining the angle that the human's gaze makes with the ball at a constant value. This method reduces a number of hard-to-monitor variables (e.g., wind speed) into a single easily-estimated variable. Modulating this variable alone results in accomplishing the task.

\begin{figure}[t]
\begin{center}
\includegraphics[width=0.99\columnwidth]{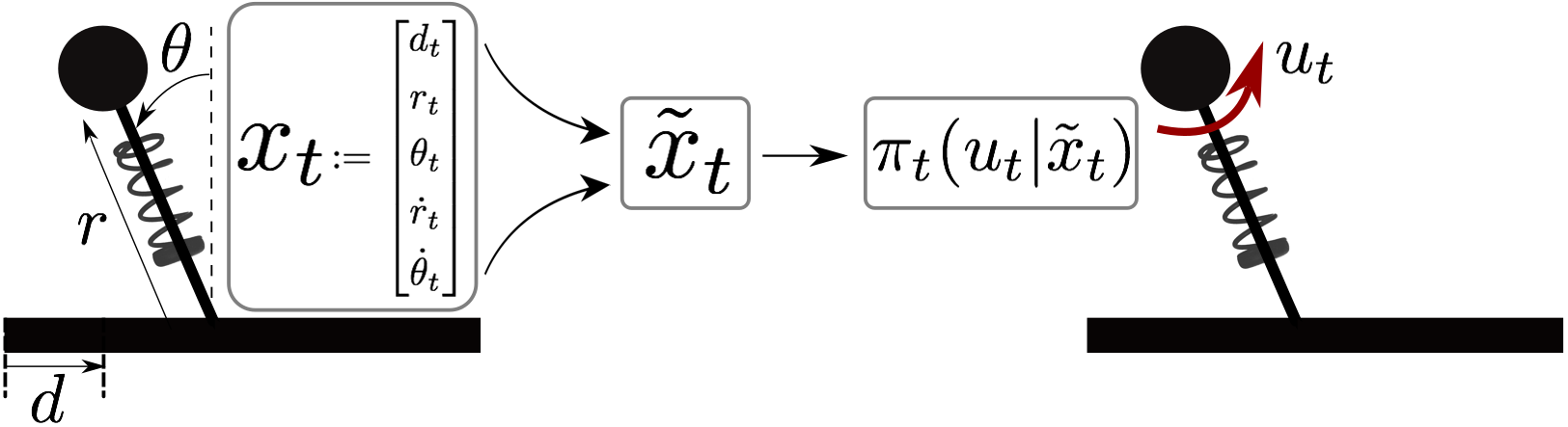}
\end{center} 
\vspace{-14pt}
\caption{A schematic of our technical approach. We seek to synthesize (offline) a minimalistic set of task-relevant variables (TRVs) $\tilde{x}_t$ that create a bottleneck between the full state $x_t$ and the control input $u_t$. These TRVs are estimated online in order to apply the policy $\pi_t$. 
We demonstrate our approach on a spring-loaded inverted pendulum model whose goal is to run to a target location. Our approach automatically synthesizes a \emph{one-dimensional} TRV $\tilde{x}_t$ sufficient for achieving this task. \label{fig:slip}}
\vspace{-18pt}
\end{figure}

The gaze heuristic example highlights the two primary advantages of using a task-driven representation. First, a control policy that uses such a representation is more efficient to employ online since fewer variables need to be estimated. Second, since only a few prominent variables need to be estimated, fewer sources of measurement uncertainty result in a more robust policy. While one can sometimes manually design task-driven representations for a given task, we currently lack a principled theoretical and algorithmic framework for \emph{automatically synthesizing} such representations. The goal of this paper is to develop precisely such an algorithmic approach. 

\textbf{Statement of Contributions.} The main technical contribution of this paper is to formulate the synthesis of task-driven representations as an optimization problem using information bottleneck theory \cite{Tishby99}. We present \emph{offline} algorithms that encode the full state of the system into a set of \emph{task-relevant variables (TRVs)} and simultaneously identify a performant policy (restricted to be a function of the TRVs) using novel iterative algorithms that exploit the structure of this optimization problem in a number of dynamical settings including discrete-state, linear-Gaussian, and nonlinear systems.
We present \emph{online} algorithms for estimating the TRVs in order to apply the control policy. We demonstrate that our approach yields policies that are robust to \emph{unmodeled} measurement uncertainty both theoretically (using results from the theory of risk metrics) and in a number of simulation experiments including running using a spring-loaded inverted pendulum (SLIP) model (Fig.\ref{fig:slip}).

\label{sec:intro}

\subsection{Related Work}

By far the most common approach in practice for controlling robotic systems with nonlinear dynamics and partially observable state is to \emph{independently} design an estimator for the \emph{full} state (e.g., a Kalman filter \cite{Thrun05}) and a controller that assumes perfect state information (e.g., designed using robust control techniques such as $H_\infty$ control \cite{Francis87}, sliding mode control \cite{Edwards98, Slotine91}, passivity-based control \cite{Ortega13, Khalil96}, or Lyapunov-based control \cite{Slotine91, Khalil96}). While this strategy is optimal for Linear-Quadratic-Gaussian (LQG) problems due to the \emph{separation principle} \cite{Anderson07}, it can produce a brittle system due to errors in the state estimate for nonlinear systems (since the separation principle does not generally hold in this setting). We demonstrate that our task-driven approach affords significant robustness when compared to approaches that perform full state estimation and control assuming the separation principle (see Section \ref{sec:examples} for numerical examples). Moreover, in contrast to traditional robust estimation and control techniques, our approach does not rely on explicit models of measurement uncertainty. We demonstrate that robustness can be achieved implicitly as a \emph{by-product} of task-driven representations.

% Our approach of designing controllers that only rely on task-relevant variables may be viewed as a form of \emph{output feedback control} \cite{}, where control is based directly on outputs (i.e., sensor measurements). However, in contrast to the standard output feedback setting, here we \emph{synthesize the output map} itself; in particular, the mapping between the state and the task-driven representation can be viewed as an ``output" corresponding to the weakest possible sensor required to achieve a performance constraint. 

Historically, the work on designing \emph{information constrained} controllers has been pursued within the networked control theory literature \cite{Nair07, Matveev09, Tanaka17}.
% Historically, this problem has been considered within the context of networked control theory, which deals with applications where sensors and actuators are physically separated and communication between the two is expensive and error-prone \cite{nair2007feedback, matveev2009estimation, Tanaka17}. 
Recently, the optimal co-design of data-efficient sensors and performant controllers has also been explored beyond network applications. One set of approaches --- inspired by the cognitive psychology concept of \emph{bounded rationality} \cite{Simon72, Gigerenzer02} --- is to limit the information content of a control policy measured with respect to a default stochastic policy \cite{Tishby11, Grau16, Todorov09, Braun11, Williams17}. Another set of examples comes from the sensor selection problem in robotics, which involves selecting a minimal number of sensors or features to use for estimating the robot's state \cite{Williams07, Tzoumas18, Carlone17}. 
While the work highlighted above shares our goal of designing ``minimalistic'' (i.e., informationally frugal) controllers, our approach here is fundamentally different. In particular, our goal is to design task-driven \emph{representations} that form abstractions which are sufficient for the purpose of control. Online estimation and control is performed purely based on this representation without resorting to estimating the full state of the system (in contrast to the work highlighted above, which either assumes access to the full state or designs estimators for it). % We believe an approach based on task-driven representations has the potential to be significantly more computationally efficient and robust than the current state of the art. 

A number of previous authors consider the construction of minimal-information representations. In information theory, this approach is typically referred to as the Information Bottleneck (IB) \cite{Tishby99}. Recently, these ideas have been co-opted for designing control policies. In \cite{Achille18}, a learning-based approach is suggested to find minimal-information state representations and control policies which use them. Our work differs in that we provide \emph{analytic} (i.e., model-based) methods for finding such representations and policies and we explicitly characterize the resulting robustness. Another branch of work considers the construction of LQG policies that achieve a performance goal while minimizing an information-theoretic quantity such as the mutual information between inputs and outputs \cite{Fox16, Fox16a} or Massey's directed information \cite{Massey90, Tanaka18}. In contrast to these works, our derivation handles nonlinear systems and also presents robustness results for the resulting controllers, which have not been discussed to our knowledge in existing literature.

The work on \emph{actionable information} in vision \cite{Soatto13, Soatto11} 
% (inspired by J. J. Gibson's ecological approach to perception \cite{Gibson79}) 
attempts to find invariant and task-relevant representations for visual tasks in the form of \emph{minimal complete representations} (minimal sufficient statistics of a full or ``complete'' representation). While highly ambitious in scope, this work has largely been devoted to studying visual decision tasks (e.g., recognition, detection, categorization). The algorithmic approach taken in \cite{Soatto13, Soatto11} is thus tied to the specifics of visual problems (e.g., designing visual feature detectors that are invariant to nuisance factors such as contrast, scale, and translation). Our goals and technical approach here are complementary. We do not specifically target visual decision problems; instead, we seek to develop a general framework that is applicable to a broad range of robotic control problems and allows us to \emph{automatically} synthesize task-driven representations  (without manually designing feature detectors).

\section{Task-Driven Representations and Controllers}

Our goal in this section is to formally define the notion of a ``task-driven representation" and formulate an optimization problem that automatically synthesizes such representations. We focus on control tasks in this paper and assume that a task is defined in terms of a (partially observable) Markov decision process ((PO)MDP). Let $\smash{x_t \in \X,\ u_t \in \U,\ y_t \in \mathcal{Y}}$ denote the full state of the system, the control input, and the sensor output at time $t$ respectively. Here the state space $\X$, input space $\U$, and output space $\mathcal{Y}$ may either be discrete or continuous. We assume that the dynamics and sensor model of the system are known and given by potentially time-varying conditional distributions $\hspace{-1mm}\smash{p_t(x_{t + 1} |u_t, x_t),\ \sigma_t(y_t | x_t)}\hspace{-1mm}$ respectively. Let $c_0(x_0, u_0), c_1(x_1, u_1), \dots, c_{T - 1}(x_{T - 1}, u_{T - 1}), c_T(x_T)$ be a sequence of cost functions that encode the robot's desired behavior. The robot's goal is to identify a control policy $\pi_t(u_t|y_t)$ that minimizes the expected value of these cost functions when the policy is executed \emph{online}, i.e. minimize $\sum_{t=0}^{T - 1} \EE_{p_t} c_t(x_t, u_t) + \EE_{p_T} c_T(x_T)$. In general, this optimization problem is infinite dimensional and challenging to solve.

The key idea behind our technical approach is to define a principled information-theoretic notion of \emph{minimality} of representations for tasks. Figure \ref{fig:slip} illustrates the main idea for doing this. Here $\smash{\tilde{x}_t \in \widetilde{\mathcal{X}}}$ are \emph{task-relevant variables} (TRVs) that constitute a \emph{representation}; one can think of the representation as a ``sketch'' of the full state that the robot uses for computing control inputs via $\pi_t(u_t | \tilde{x}_t)$. Ideally, such a representation filters out all the information from the state that is not relevant for computing control inputs --- avoiding the introduction of unnecessary estimation error. A minimalistic representation $\tilde{x}_t$ should thus create a \emph{bottleneck} between the full state and the control input. We make this notion precise by leveraging the theory of information bottlenecks \cite{Tishby99} and finding a stochastic mapping $q_t(\tilde{x}_t | x_t)$ that minimizes the \emph{mutual information} between $x_t$ and $\tilde{x}_t$ 
\begin{equation}
\mathbb{I}(x_t; \tilde{x}_t) \coloneqq \KL(p_t(x_t, \tilde{x}_t) \| p_t(x_t) q_t(\tilde{x}_t)),
\end{equation}
while ensuring the policy performs well (i.e., achieves low expected cost). Here, $\KL(\cdot \| \cdot)$ represents the Kullback-Leibler (KL) divergence between two distributions. % and the mutual information thus measures how far the joint distribution $p_t(x_t, \tilde{x}_t)$ is from the distribution $p_t(x_t) p_t(\tilde{x}_t)$ where $x_t$ and $\tilde{x}_t$ are \emph{independent}. 
Intuitively, minimizing the mutual information corresponds to designing TRVs $\tilde{x}_t$ that are as independent of the state $x$ as possible. The map $q_t(\tilde{x}_t | x_t)$ is thus ``squeezing out'' all the irrelevant information from the state while maintaining enough information for choosing good control inputs. This representation formalizes our notion of a \emph{task-driven representation}. 

Formally, we pose the problem of finding task-driven representations as the following \emph{offline} optimization problem, which we refer\footnote{Technically, $\OPT$ is a kind of rate-distortion problem, not an information bottleneck problem, as the constraint is not specified using a divergence as a distortion function. However, $q_t(\tilde{x}_t | x_t)$ limits the flow of information from $x_t$ to $u_t$ so we use the term bottleneck as a conceptual aid.} to as $\OPT$:
\begin{flalign*}
&& \underset{p_t, q_t, \pi_t}{\textrm{min}} \hspace*{1cm} \sum_{t=0}^{T} \Big{[} \EE_{p_t} c_t(x_t, u_t)  + \frac{1}{\beta} \mathbb{I}(x_t; \tilde{x}_t) \Big{]}. &&& (\OPT)
\end{flalign*}
We note that the unconstrained problem $\OPT$ is equivalent to a constrained version where the mutual information is minimized subject to a constraint on the expected cost of the full-state MDP and $\beta$ is the corresponding Lagrange multiplier. A heuristic approach for choosing an appropriate value for $\beta$ is discussed in \cref{sec:examples}.

So far, we have limited our discussion to the case where the robot has access to the full state of the system. However, the real benefit of the task-driven perspective is evident in the partially observable setting, where the robot only indirectly observes the state of the system via sensor measurements $y_0, \dots, y_t$, where $y_t \in \mathcal{Y}_t$. We denote the probability of observing a particular measurement in a given state by $\sigma_t(y_t | x_t)$. The prevalent approach for handling such settings is to design an estimator for the full state $x_t$. Our key idea here is to perform estimation for the TRVs $\tilde{x}_t$ instead of $x_t$. Specifically, our overall approach has two phases:

{\bf Offline.} Synthesize the maps $\{q_t, \pi_t\}_{t = 0}^{T - 1}$ by solving $\OPT$.

{\bf Online.} Estimate current TRV $\tilde{x}_t$ using sensor measurements and use this estimate to compute inputs via $\pi_t(u_t | \tilde{x}_t)$.

Perhaps the clearest benefit of our approach is the fact that the representation $\tilde{x}_t$ may be significantly lower-dimensional than the full state of the system; this can lead to significant reductions in online computations. Another advantage of the task-driven approach is robustness to estimation errors. To see this, let $p_t(x_t, \tilde{x}_t, u_t)$ denote the joint distribution over the state, representation, and inputs at time $t$ that results when we apply the policy obtained by solving Problem $\OPT$ in the fully observable setting. Now, let $\tilde{p}_t(x_t, \tilde{x}_t, u_t)$ denote the distribution for the partially observable setting, i.e., when we estimate $\tilde{x}_t$ online using the robot's sensor measurements and use this estimate to compute control inputs.

% us consider the partially observable setting where we estimate $\tilde{x}_t$ online using the robot's sensor measurements and use this estimate to compute control inputs. Let $\tilde{p}_t(x_t, \tilde{x}_t, u_t)$ denote the distribution that results from applying such a policy.

\begin{theorem}
\label{thm:robustness}
Let $\tilde{p}_t(x_t, \tilde{x}_t, u_t)$ be the distribution resulting from any estimator that satisfies the following condition:
\begin{align}
\label{eq:KL condition}
&\KL(\tilde{p}_t(x_t, \tilde{x}_t, u_t) \| p_t(x_t, \tilde{x}_t, u_t)) \\
&\leq \frac{1}{\beta}\KL(p_t(x_t, \tilde{x}_t) \| p_t(x_t) q_t(\tilde{x}_t)). \nonumber
\end{align}
Then, we have the following upper bound on the total expected cost:
\begin{equation}
\label{eq:cost bound}
\sum_{t=0}^{T} \EE_{\tilde{p}_t} c_t(x_t, u_t) \leq \sum_{t=0}^{T} \Big{[} \rho \big{(} c_t(x_t, u_t) \big{)} + \frac{1}{\beta}\mathbb{I}(x_t; \tilde{x}_t) \Big{]},
\end{equation}
where $\rho$ is the \emph{entropic risk metric} \cite[Example 6.20]{Shapiro09}:
\begin{equation}
\label{eq:entropic risk}
\rho \big{(} c_t(x_t, u_t) \big{)} \coloneqq \log \Big{[} \EE_{p_t} \exp(c_t(x_t, u_t))\Big{]}.
\end{equation}
\end{theorem}
\vspace{10pt}
\begin{proof} 
%\hspace{-5mm}
\iftoggle{extended}{%
By the well-known Donsker-Varadhan change of measure formula \cite[Theorem 2.3.2]{Gray11}, we have:
\begin{align}
\label{eq:DV ineq}
\EE_{\tilde{p}_t} c_t(x_t, u_t) \leq &\log \Big{[} \EE_{p_t}  \exp(c_t(x_t, u_t)) \Big{]} \nonumber \\
+&\KL(\tilde{p}_t(x_t, \tilde{x}_t, u_t) \| p_t(x_t, \tilde{x}_t, u_t)). 
\end{align}
Then, using condition \eqref{eq:KL condition} and inequality \eqref{eq:DV ineq}, we obtain:
\begin{align}
\EE_{\tilde{p}_t}\ c_t(x_t, u_t) &\leq \log \Big{[} \EE_{p_t} \exp(c_t(x_t, u_t))\Big{]} \nonumber \\
&+\frac{1}{\beta}\KL(p_t(x_t, \tilde{x}_t) \| p_t(x_t) q_t(\tilde{x}_t)) \\
&= \rho \big{(} c_t(x_t, u_t) \big{)} + \frac{1}{\beta}\mathbb{I}(x_t; \tilde{x}_t).
\end{align}
Summing over time gives us the desired result.
}
{
The proof follows from the Donsker-Varadhan change of measure formula \cite[Theorem 2.3.2]{Gray11} and is presented in the extended version of this paper \cite{Pacelli18}.
}
\end{proof}

Intuitively, this theorem shows that \emph{any} estimator for $\tilde{x}_t$ (in the partially observable setting) that results in a distribution $\tilde{p}_t(x_t, \tilde{x}_t, u_t)$ that is ``close enough'' to the distribution $p_t(x_t, \tilde{x}_t, u_t)$ in the fully observable case (i.e., when their KL divergence is less than $\frac{1}{\beta}$ times the KL divergence between $p_t(x_t, \tilde{x}_t)$ and the joint distribution $p_t(x_t) q_t(\tilde{x}_t)$ over $x_t$ and $\tilde{x}_t$ that results when $\tilde{x}_t$ is assumed to be \emph{independent} of $x_t$), the expected cost of the controller in the partially observable case is \emph{guaranteed} to be bounded by the right hand side (RHS) of \eqref{eq:cost bound}. Notice that this RHS is similar to the cost function of $\OPT$. In particular, the expected value operator is a \emph{linearization} of the entropic risk metric\footnote{This risk metric has a long history in robust control (including a close link to $H_\infty$ control) \cite{Whittle81, Whittle02, Glover88}. We note, however, that it can sometimes be conservative in cases with rare, bad events.} $\rho$. By solving $\OPT$, we are minimizing (a linear approximation of) an upper bound on the expected cost even when our state is only partially observable (as long as our estimator for $\tilde{x}_t$ ensures condition \eqref{eq:KL condition}). Once $\OPT$ is solved, we can use Theorem \ref{thm:robustness} to obtain a robustness bound by evaluating the RHS of \eqref{eq:cost bound}.

\section{Algorithms for Synthesizing Representations}
\label{sec:solution}
In this section, we outline our approach for solving $\OPT$ offline.
We note that $\OPT$ is non-convex in general. While one could potentially apply general non-convex optimization techniques such as gradient-based methods, computing gradients quickly becomes computationally expensive due to the large number of decision variables involved (even in the setting with finite state and action spaces, we have decision variables corresponding to $q_t(\tilde{x}_t|x_t)$ and $\pi_t(u_t|\tilde{x}_t)$ for every possible value of $x_t, \tilde{x}_t$ and $u_t$ at every time step). Our key insight here is to exploit the structure of $\OPT$ to propose an efficient iterative algorithm in three different dynamical settings: discrete, linear-Gaussian, and nonlinear-Gaussian. These settings are particularly convenient to work with because they allow the objective of $\OPT$ to be \emph{computed in closed-form}. In each setting, the algorithm iterates over the following three steps:
\begin{enumerate}
	\item Fix $\{q_t, \pi_t\}_{t = 0}^{T - 1}$ and solve for $\{p_t\}_{t = 0}^T$ using the forward dynamical equations.
	\item Fix $\{p_t\}_{t = 0}^T, \{\pi_t\}_{t = 0}^{T - 1}$ and solve for $\{q_t\}_{t = 0}^{T - 1}$ by satisfying necessary conditions for optimality.
	\item Fix $\{p_t\}_{t = 0}^T, \{q_t\}_{t = 0}^{T - 1}$ and solve for $\{\pi_t\}_{t = 0}^{T - 1}$  by solving a convex optimization problem. 
\end{enumerate}
In our implementation, we iterate over these steps until convergence (or until an iteration limit is reached). This is a common strategy employed in solving similar kinds of MDPs with information-theoretic objectives \cite{Fox16, Fox15}. While we cannot currently guarantee convergence, our iterative procedure is extremely efficient (since all the computations above can be performed either in closed-form or via a convex optimization problem) and produces good solutions in practice (see Section \ref{sec:examples}). We describe instantiations of each step for three different dynamical settings below.

\subsection{Discrete Systems}

In order to solve Step 1, note that the forward dynamics of the system for fixed $q_t(\tilde{x}_t | x_t), \pi_t(u_t | \tilde{x}_t)$ are given by:
\begin{flalign}
	p_t(x_{t + 1} | x_t) &= \sum_{u, \tilde{x}}p_t(x_{t + 1} | x_t, u)\pi_t(u | \tilde{x})q_t(\tilde{x} | x_t), \nonumber \\
	p_{t + 1}(x_{t + 1}) &= \sum_{x} p_t(x_{t + 1} | x) p_t(x). \label{eq:forward eq}
\end{flalign}
The Lagrangian functional for $\OPT$ is $\mathcal{L} = \sum_{t = 0}^{T} \mathcal{L}_t$ where
\begin{footnotesize} 
\begin{flalign*}
\mathcal{L}_t &= \sum_{\tilde{x}, u, x}c_t(x, u)\pi_t(u|\tilde{x})q_t(\tilde{x}|x)p_t(x) \\
	&- \sum_{x'}\nu_{t + 1}(x')\left(p_{t + 1}(x') - \sum_{x, u, \tilde{x}}p_t(x'|x, u)\pi_t(u|\tilde{x})q_t(\tilde{x}|x)p_t(x)\right) \\
	&+ \frac{1}{\beta}\sum_{x, \tilde{x}}q_t(\tilde{x}|x)p_t(x)\log\left(\frac{q_t(\tilde{x}|x)}{q_t(\tilde{x})}\right). \nonumber
\end{flalign*}
\end{footnotesize}
where $\nu_t(x_t)$ are Lagrange multipliers. The Lagrange multipliers that normalize distribution variables are omitted since they do not contribute to the analysis. The following proposition demonstrates the structure of $q_t(\tilde{x}_t | x_t)$ using the first-order necessary condition (FONC) for optimality \cite{Boyd04}.

\begin{theorem}
	\label{thm:boltzmann}
	A necessary condition for $q_t(\tilde{x} | x)$ to be optimal for $\OPT$ is that
	\begin{flalign}
		q_t(\tilde{x}_t | x_t) = \frac{q_t(\tilde{x}_t)\exp\left(-\beta\mathbb{E}(\nu_{t + 1} + c_t | x_t, \tilde{x}_t)\right)}{Z_t(x_t)}, \label{eq:code given state}
	\end{flalign}
	where $\nu_T(x_T) = c_T(x_T)$ and
	\begin{flalign}
		\nu_t(x_t) &= \EE(c_t + \nu_{t + 1}| x_t) + \frac{1}{\beta}\KL\left(q_t(\tilde{x}_t|x_t)\|q_t(\tilde{x}_t)\right),\label{eq:value func}\\
		Z_t(x_t) &= \sum_{\tilde{x}}q_t(\tilde{x})\exp\left(-\beta[\EE(c_t + \nu_{t + 1}| x_t, \tilde{x}))]\right). \label{eq:partition func} 
	\end{flalign}
\end{theorem}

\begin{proof}
%Equation \eqref{eq:value func} is derived by setting the functional derivative $\delta \mathcal{L}/\delta p_t(x_t) = 0$ and solving for $\nu_t$. Repeating this process for $\delta \mathcal{L}/\delta q_t(x_t | \tilde{x}_t)$ yields \eqref{eq:code given state}. Equation \eqref{eq:partition func} follows from the fact that $q_t(\tilde{x}_t | x_t)$ is a probability distribution and must thus be normalized. A complete proof can be found in the extended version \cite{Pacelli18}.
\iftoggle{extended}{See \cref{apdx:boltzmann}.}{Equation \eqref{eq:value func} is derived by setting the functional derivative $\delta \mathcal{L}/\delta p_t(x_t) = 0$ and solving for $\nu_t$. Repeating this process for $\delta \mathcal{L}/\delta q_t(x_t | \tilde{x}_t)$ yields \eqref{eq:code given state}. Equation \eqref{eq:partition func} is the normalization of $q_t(\tilde{x}_t | x_t)$. Proof is provided in \cite{Pacelli18}.}
\end{proof}

This proposition demonstrates that $q_t(\tilde{x}_t | x_t)$ is a \emph{Boltzmann distribution} with $Z_t(x_t)$ as the \emph{partition function} and $\beta$ playing the role of inverse temperature. In order to solve Step 2, we simply evaluate the closed-form expression \eqref{eq:code given state}.

It is easily verified that the function $\nu_t(x_t)$ is the \emph{cost-to-go function} for $\OPT$. Thus $\OPT$ can be written as a dynamic programming problem using $\nu_t(x_t)$:
\begin{flalign}
	\underset{q_t, \pi_t}{\textrm{min}} \quad \EE \left(c_t + \nu_{t + 1} + \frac{1}{\beta}\KL\big(q_t(\tilde{x}_t|x_t)||q_t(\tilde{x}_t)\big)\right). && \hspace{-1.5mm}(\DP)	\nonumber
\end{flalign}

This allows us to solve Step 3. In particular, when $p_t(x_t), q_t(
\tilde{x}_t | x_t)$ are fixed, $\DP$ is a linear programming problem in $\pi_t$ and can thus be solved efficiently. 

%Since the $p_t, q_t, \pi_t$ are finite dimensional, the RHS of \cref{eq:forward eq} and \cref{eq:code given state} can be computed directly allowing for closed-form solutions to steps 1 and 2. Moreover, when $q_t, p_t$ are fixed, $\DOPT$ is a linear programming problem in $\pi_t$ and can be solved efficiently for step 3.

\subsection{Linear-Gaussian Systems with Quadratic Costs}
\label{sec:LG}

A discrete-time linear-Gaussian (LG) system is defined by the transition system
\begin{flalign}
	x_{t + 1} &= A_t x_t + B_t u_t + \epsilon_t, \quad \epsilon_t \sim N(0, \Sigma_{\epsilon_t}),\label{eq:lg}
\end{flalign}
where $\mathcal{X} = \RR^n$, $\mathcal{U} = \RR^m$ and $x_0 \sim N(\bar{x}_0, \Sigma_{x_0})$. We assume that the cost function is quadratic:
\begin{small}
\begin{equation}
	c_t(x, u) \coloneqq \frac{1}{2}\transp{(x - g_t)}Q_t(x - g_t) + \frac{1}{2}\transp{(u - w_t)} R_t (u - w_t),\nonumber
\end{equation}
\end{small}
with $Q_t,\ R_t \succeq 0, R_T = 0$. We explicitly parameterize the TRVs and control policy as:
\begin{equation}
	\tilde{x}_t = C_t x_t + a_t + \eta_t, \quad u_t = K_t\tilde{x}_t + h_t,
\end{equation}
where the random variable $\eta_t \sim N(0, \Sigma_{\eta_t})$ is additive process noise. This structure dictates that $p_t(x_t),\  q_t(\tilde{x}_t | x_t)$ are Gaussians $N(\bar{x}_t, \Sigma_{x_t}),\ N(\bar{\tilde{x}}_t, \Sigma_{\tilde{x}_t})$ respectively, with $\bar{\tilde{x}}_t = C_t \bar{x}_t + a_t,\ \Sigma_{\tilde{x}_t} = C_t \Sigma_{x_t} \transp{C}_t + \Sigma_{\eta_t}$. This allows for both Steps 1 and 2 to be computed in closed form. The latter is presented in the following theorem.
\begin{theorem}
\label{thm:lg sln}
Define the notational shorthand $G_t \coloneqq \transp{C_t}\inv{(C_t\Sigma_{x_t}\transp{C}_t + \Sigma_{\eta_t})}C_t,\ M_t \coloneqq (A_t + B_t K_t C_t)$. For the LG system, the necessary condition \cref{eq:code given state} is equivalent to the conditions
\begin{flalign}
%\small
	C_t = &-\beta \Sigma_{\eta_t}\transp{K}_t \transp{B}_t P_{t + 1} A_t, \nonumber\\
	a_t = &-\Sigma_{\eta_t}\big(\beta \transp{K}_t \transp{B}_t (b_{t + 1} + P_{t + 1} B_t h_t)\label{eq:linnec}\\
	&+\beta \transp{K}_t R_t (h_t - w_t) - \Sigma^{-1}_{\tilde{x}_t}\bar{\tilde{x}}_t\big), \nonumber\\
	\inv{\Sigma}_{\eta_t} = &\Sigma^{-1}_{\tilde{x}_t} + \beta \transp{K}_t  (\transp{B}_t P_{t + 1} B_t + R_t) K_t,\nonumber
\end{flalign}
where the cost-to-go function is the recursively defined quadratic function $\nu_t(x) = \frac{1}{2}\transp{x}_tP_tx_t + \transp{b_t}x_t + \mathrm{constant}$ with values $P_{T} = Q_{T + 1},\ b_{T} = -Q_{T} g_{T}$ and
\begin{flalign}
	P_t = &Q_t + \beta^{-1}G_t + \transp{C}_t \transp{K}_t R_t K_t C_t + \transp{M_t}P_{t + 1}M_t,\nonumber\\
	b_t = &\transp{M_t} P_{t + 1} B_t (h_t + K_t a_t) - Q_t g_t - \inv{\beta} G_t \bar{x}_t\nonumber\\ 
	&+ \transp{C}_t \transp{K}_t R_t (K_t a_t + h_t - w_t) + \transp{M_t}b_{t + 1}. \label{eq:value quad}
\end{flalign}
\end{theorem}
\begin{proof}
\iftoggle{extended}{See \cref{apdx:lg sln}.}{The KL-divergence term in \eqref{eq:value func} is quadratic: $\KL\left(q(\tilde{x}|x_t)||q_t(\tilde{x})\right) = \frac{1}{2}\transp{(\bar{x}_t - x_t)} G_t(\bar{x}_t - x_t)$ up to a constant. Since $c_t$ is quadratic in $x_t$, the form for $\nu_t(x_t)$ is derived by backward induction starting with $\nu_T(x_T) = \EE(c_T)$. The forms for $C_t, a_t, \Sigma_{\eta_t}$ are derived by taking the logarithm of both sides of \cref{eq:code given state}, plugging \cref{eq:value quad} in for $\nu_t(x_t)$, completing the square, and exponentiating both sides to produce a Gaussian. The constant term in \cref{eq:value quad} is collected into $Z_t(x_t)$. A complete proof can be found in \cite{Pacelli18}.}
\end{proof}

\iftoggle{extended}{
Finally, when $q_t, p_t$ are fixed, $\DP$ is the unconstrained convex quadratic program $\min_{K_t, h_t} J(K_t, h_t)$ where
\begin{flalign}
J(&K_t, h_t) \coloneqq \frac{1}{2} \transp{(K_t \bar{\tilde{x}}_t + h_t - w_t)} R_t (K_t \bar{\tilde{x}}_t + h_t - w_t)\nonumber\\
	&+ \frac{1}{2}\tr{\Sigma_{\tilde{x}_t} \transp{K}_t R_t K_t} + \frac{1}{2}\transp{\bar{x}}_t \transp{A}_t P_{t + 1}( A_t \bar{x}_t +  2B_t h_t)\nonumber \\
	&+\frac{1}{2} \transp{\bar{\tilde{x}}}_t \transp{K}_t \transp{B}_t P_{t + 1} B_t (K_t \bar{\tilde{x}}_t + 2h_t) + \frac{1}{2} \transp{f}_t \transp{B}_t P_{t + 1} B_t h_t\nonumber\\
	&+ \frac{1}{2} \transp{\bar{x}}_t (\transp{A}_t P_{t + 1} B_t  K_t C_t + \transp{C} K_t B_t P_{t + 1} A_t) \bar{x}_t\nonumber\\
    & + \transp{\bar{x}}_t \transp{A} P_{t + 1} B_t K_t a_t + \frac{1}{2}\tr{\Sigma_{\tilde{x}_t} \transp{K} \transp{B} P_{t + 1} B_t K_t}\nonumber\\
    &+ \transp{b}_{t + 1} (A_t \bar{x}_t + B_t K_t \bar{\tilde{x}}_t + B_t h_t)\\
    &+ \frac{1}{2} \tr{\Sigma_{x_t} \transp{A}_t P_{t + 1} A_t} + \frac{1}{2}\tr{\Sigma_{x_t}\transp{C}_t \transp{K}_t \transp{B} P_{t + 1} A_t}\nonumber\\
    &+ \frac{1}{2}\tr{\Sigma_{x_t} \transp{A}_t P_{t + 1} B_t K_t C_t}.\nonumber
\end{flalign}
This program can be solved very efficiently (e.g., using active-set or interior point methods) \cite{Wright99}.
}{
Finally, when $q_t, p_t$ are fixed, $\DP$ is the unconstrained convex quadratic program with decision variables $K_t, h_t$, and can be solved very efficiently \cite{Wright99}.
}
%\begin{flalign*}
%	\min_{\substack{K_t, h_t\\m_1, m_2}} &\transp{m_1} P_{t + 1} m_1 + \transp{b}_{t + 1}m_1 + \transp{\tilde{x}}_t \transp{K}_t R_t K_t \tilde{x}_t + \mathrm{const}.\\
%	\textrm{s.t.}\ & m_1 = A_t x_t + B_t m_2,\ m_2 = K_t \tilde{x}_t + h_t
%\end{flalign*}

\subsection{Nonlinear-Gaussian Systems}
\label{sec:ilqg}
When the dynamics are nonlinear-Gaussian (NLG), i.e. when \cref{eq:lg} is changed to
\begin{flalign}
	x_{t + 1} &= f(x_t, u_t) + \epsilon_t, \quad \epsilon_t \sim N(0, \Sigma_{\epsilon_t}),\label{eq:nlg}
\end{flalign}
minimizing $\OPT$ is challenging due to $p_t(x_t)$ no longer being Gaussian. We tackle this challenge by leveraging our results for the LG setting and adapting the iterative Linear Quadratic Regulator (iLQR) algorithm \cite{Todorov05, Li04, Jacobson70}.

Given an initial nominal trajectory $\{\hat{x}_t, \hat{u}_t\}_{t = 0}^{T}$, the matrices $\{A_t, B_t\}_{t = 0}^{T - 1}$ are produced by linearizing $f(x_t, u_t)$ along the trajectory. %, i.e. $A_t = D_x f(x_t, u_t), B_t = D_u f(x_t, u_t)$ where $D_\cdot$ is the Jacobian operator. 
The pair $(A_t, B_t)$ describes the dynamics of a perturbation $\delta x_t = x_t - \hat{x}_t$ in the neighborhood of $x_t$ for a perturbed input $\delta u_t = u_t - \hat{u}_t$ in the neighborhood of $u_t$:
\begin{flalign}
\delta x_{t+1} = A_t \delta x_t + B_t \delta u_t + \epsilon_t, \quad \delta x_0 \sim N(0, \Sigma_{x_t}). \label{eq:delta lg}
\end{flalign} 

We compute (a quadratic approximation of) the perturbation costs $\delta c_t(\delta x, \delta u) \coloneqq c_t(\hat{x}_t + \delta x, \hat{u}_t + \delta u)$ subject to \cref{eq:delta lg}. We can then apply the solution method outlined in Section \ref{sec:LG} to search for an optimal $\{\delta x_t, \delta u_t\}_{t = 0}^{T}$. We then update the nominal state and input trajectories to $\{\hat{x}_t + \delta x_t,\hat{u}_t + \delta u_t\}_{t = 0}^{T}$ and repeat the entire process until the nominal trajectory converges.

\section{Online Estimation and Control}

Once the task-driven representation and policy have been synthesized offline, we can leverage them for computationally-efficient and robust online control. The key idea behind our online approach is to use the robot's sensor measurements $\{y_i\}_{i = 1}^t$ to \emph{only} estimate the TRVs $\tilde{x}_t$. Once $\tilde{x}_t$ has been estimated, the control policy $\pi_t(u_t | \tilde{x}_t)$ can be applied. This is in stark contrast to most prevalent approaches for controlling robotic systems, which aim to accurately estimate the full state $x_t$. We describe our online estimation approach below. 

We maintain a belief distribution $\textrm{bel}(\tilde{x}_t)$ over the TRV-space $\widetilde{\mathcal{X}}$ and update it at each time $t$ using a Bayes filter \cite{Thrun05}. Specifically, we perform two steps every $t$:
\begin{enumerate}
\item \textbf{Process Update.} The system model is used to update the belief-state to the current time step: $\overline{\textrm{bel}}(\tilde{x}_{t}) = \sum_{\tilde{x}_{t-1}} q_{t-1}(\tilde{x}_{t} | \tilde{x}_{t - 1}, u_{t - 1}) \textrm{bel}(\tilde{x}_{t - 1})$.
\item \textbf{Measurement Update.} The measurement model is used to integrate the observation $y_t$ into the belief-state: $\textrm{bel}(\tilde{x}_t) \propto \sigma_t(y_t | \tilde{x}_t) \overline{\textrm{bel}}(\tilde{x}_t)$.
\end{enumerate}
To apply this filter, the distributions $q_t(\tilde{x}_{t + 1} | \tilde{x}_t,  u_t)$ and $\sigma_t(y_t | \tilde{x}_t)$ are precomputed \emph{offline}. Bayes' theorem states $p_t(x_t | \tilde{x}_t) = q_t(\tilde{x}_t | x_t) p_t(x_t)/q_t(\tilde{x}_t)$. Consequently,
\begin{flalign*}
\small
p_t(x_{t + 1} | \tilde{x}_t, u_t) &= \sum_{x_t} p_t(x_{t + 1} | u_t, x_t) p_t(x_t | \tilde{x}_t),\\
q_{t}(\tilde{x}_{t + 1} | \tilde{x}_t, u_t) &= \sum_{x_{t + 1}} q_{t}(\tilde{x}_{t + 1} | x_{t + 1})p_t(x_{t + 1} | \tilde{x}_t, u_t),\\
\sigma_t(y_t | \tilde{x}_t) &= \sum_{x_t} \sigma_t(y_t | x_t) p_t(x_t | \tilde{x}_t).
\end{flalign*}
\iftoggle{extended}{%
In the discrete case, the above equations can be evaluated directly. In the LG case with sensor model,
\begin{flalign*}
	y_t = D_t x_t + \omega_t, \quad \omega_t \sim N(0, \Sigma_{\omega_t}),
\end{flalign*}
the measurement and update steps take the form of  the traditional Kalman updates applied to an LG system induced on the TRV-space by $q_t$ and the system dynamics. This structure is elucidated in the following theorem.
\begin{theorem}
	\label{thm:kalman}
	The measurement and process updates for a Bayesian filter on the TRVs $\{\tilde{x}_t\}_{t = 0}^T$ in the LG case are the Kalman filter measurement and process updates for the LG system
	\begin{flalign}
		\tilde{x}_{t + 1} &= \tilde{A}_t \tilde{x}_t + \tilde{B}_t u_t + \tilde{r}_t + \tilde{\epsilon}_t,\\
		\tilde{y}_t &= \tilde{D}_t \tilde{x}_t + \tilde{\omega}_t,
	\end{flalign}
	where
	\begin{flalign*}
		\tilde{A}_t &\coloneqq C_{t + 1} A_t \Sigma_{x_t} \transp{C}_t \inv{\Sigma}_{\tilde{x}_t}, & \tilde{C}_t &\coloneqq D_t \Sigma_{x_t} \transp{C}_t \inv{\Sigma}_{\tilde{x}_t},\\
		\tilde{B}_t &\coloneqq C_{t + 1} B_t, & \tilde{\omega}_t &\sim N(0, \Sigma_{y_t | \tilde{x}_t}),\\
		\tilde{r}_t &\coloneqq -\tilde{A}_t \bar{\tilde{x}}_t + C_{t + 1} A_t \bar{x}_t + a_{t + 1}, & \tilde{\epsilon}_t &\sim N\left(0, \Sigma_{\tilde{x}_t|\tilde{x}_t, u_t}\right).
	\end{flalign*}
\end{theorem}
\begin{proof}
	See \cref{apdx:kalman}.
\end{proof}
}{
In the discrete case, the above equations can be evaluated directly. In the LG case, $p_t(x_t | \tilde{x}_t)$ is given by the minimum mean squared error estimate of $x_t$ given $\tilde{x}_t$. If $y_t \sim N(D_t \bar{x}_t, \Sigma_{\omega_t}),\ D_t \in \RR^{l \times n}$, then $\sigma_t(y_t | \tilde{x}_t)$ and the mean of $q_t(\tilde{x}_{t + 1} | \tilde{x}_t, u_t)$ are described by the equations
\begin{flalign*}
\small
\bar{\tilde{x}}_{t + 1} &= C_{t + 1}A_t [\bar{x}_t + \Sigma_{x_t} \transp{C}_t \Sigma_{\tilde{x}_t | x_t}(\tilde{x}_t - \bar{\tilde{x}}_t)] + C_{t + 1} B_t \bar{u}_t,\\
\bar{y}_t &= D_t \bar{x}_t + D_t \Sigma_{x_t} \transp{C}_t \Sigma_{\tilde{x}_t | x_t}(\tilde{x}_t - \bar{\tilde{x}}_t),\\
\Sigma_{y_t} &= D_t \Sigma_{\tilde{x}_t | x_t}\transp{C}_t \Sigma_{\eta_t}^{-1} \transp{D}_t + \Sigma_{\omega_t}.
\end{flalign*}
This system is a partially observable linear-Gaussian system with process and measurement noise induced by the randomness of both $\tilde{x}_t$ and $y_t$. The process and measurement updates are standard Kalman filter updates applied to this system.
}

In the NLG case, we use an extended Kalman filter over $\delta \tilde{x}_t$ (the perturbed TRV). Specifically, we use the linearized dynamics of $\delta \tilde{x}_t$ and apply the  LG systems approach. 

Given a belief $\textrm{bel}(\tilde{x}_t)$, we compute the control input by sampling $u_t \sim \pi_t(u_t|\tilde{x}_t^\star)$, where $\tilde{x}_t^\star$ is the maximum likelihood TRV $\tilde{x}_t^\star \coloneqq \max_{\tilde{x}_t} \ \textrm{bel}(\tilde{x}_t)$. Alternatively, one can sample the TRV from $q_t(\tilde{x}_t | x_t)$, but the MLE method is similar to how many Bayesian filters are implemented (e.g. Kalman filters) and allows for a more direct comparison between approaches.

\section{Examples}
\label{sec:examples}

In this section, we demonstrate the efficacy of our task-driven control approach on both a discrete-state scenario and a SLIP model. To select the value of $\beta$, the algorithm is run 10 times with $\beta$ set to evenly spaced values in a listed interval. The controller used is the one with the lowest value of $\beta$ subject to the expected cost of the controller being below a particular value. This choice selects the performant controller with the least state information in the TRVs.

\subsection{Lava Problem}

The first example (\cref{fig:lava}), adapted from \cite{Florence2017} demonstrates a setting where the separation principle \emph{does not hold}. If the robot's belief distribution is $\transp{[0.3, 0.4, 0.0, 0.3, 0.0]}$ while residing in state 4, the optimal action corresponding to the MLE of the state is to move right --- not the optimal left.

Our algorithm was run with three TRVs and $\beta \in [0.001, 1]$ for 30 iterations. The value $\beta = 0.001$  was found to give a negative expected cost. The robot's initial condition was sampled from the aforementioned belief distribution. Online, the robot was modeled to have a faulty sensor localizes to the correct state with a probability of 0.5, with a uniformly random incorrect state returned otherwise.

\cref{fig:lava_results} compares our algorithm's performance with a separation principle approach, i.e. solving the MDP with perfect state information and then performs MLE for the state online. Interestingly, our algorithm produces a \emph{deterministic policy} that moves the robot left three times --- ensuring it is in state 1 --- then moves right twice and remains in the goal. With this policy, it is impossible for the robot to enter the lava state, producing  \emph{a more performant, lower variance trajectory than the separation principle-based solution under measurement noise}. The solution is deterministic at low values of $\beta$ because the distribution in \cref{eq:code given state} is almost uniform --- thereby requiring the policy to be effectively open-loop.%$\{\pi_t\}_{t = 0}^{T}$ to be essentially open-loop to provide an effective policy. %Note that despite the policy being deterministic, the cumulative cost is still random due to the random initial condition.

\begin{figure}[t]
\vspace{7pt}
\begin{center}
\includegraphics[width=0.75\columnwidth]{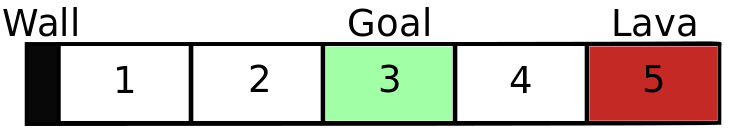}
\end{center} 
\vspace{-12pt}
\caption{The \emph{lava scenario} \cite{Florence2017} consists of five states connected in a line. The robot is allowed to move one step in either direction unless it enters the lava pit on the far right, which is absorbing. The robot receives a reward of 5 points upon entering the goal state and a penalty reward of -1 point otherwise. The terminal rewards are 10 points when in the goal and -10 points when in the lava.\label{fig:lava}}
\vspace{-18pt}
\end{figure}

\subsection{SLIP Model Problems}

Next, we apply the NLG variant of our algorithm to the SLIP model \cite{M'Closkey1993, Schwind1995, Geyer2005}, which is depicted in \cref{fig:slip}. The SLIP model is common in robotics for prototyping legged locomotion controllers. It consists of a single leg whose lateral motion is derived from a spring/piston combination. At touchdown, the state of the robot is given by $\transp{[d, \theta, \dot{r}, \dot{\theta}]}$ where $d$ is displacement of the head from the origin, $\theta$ is the touchdown angle, and $\dot{r}, \dot{\theta}$ are the radial and angular velocities. The system input is $\Delta \theta$, the change in the next touchdown angle. The parameters are the head mass, $M = 1$, the spring constant, $k = 300$, gravity $g = 9.8$, and leg length $r_{max} = 1$. Despite the model's simplicity, the touchdown return map eludes a closed-form description, so MATLAB's \texttt{ode45} is used to compute and linearize the return map.

The goal is to place the head of the robot at $d = 3.2$ after three hops. This experiment is based on a set of psychology experiments that examined the cognitive information used by humans for foot placement while running \cite{Warren1986, Warren2012}. Our NLG algorithm was run with $\beta \in [1, 200]$, control cost matrices $R_t = 10$ for all $t$, and a terminal state cost as the squared distance of the robot from $d = 3.2$. The initial distribution was Gaussian with mean $\transp{[0, 0.3927, -3.273, -6.788]}$, which is in the vicinity of a fixed point of the return map, and covariance $10^{-3}I$. The process covariance was $\Sigma_{\epsilon_t} = 10^{-4}\mathrm{diag}(1, 0.1, 0.5, 0.5)$. % The algorithm updated its nominal trajectory once after 5 iterations and then ran to convergence. 
Here, $\beta = 23.11$.

The results for our simulation are shown in \cref{fig:slip_results}. The algorithm is compared with iLQG solutions with correct and incorrect measurement models. The believed measurement model was a noisy version of the state with covariance $\Sigma_{\omega_t} = 10^{-4} I$ while the actual measurement model used $\Sigma_{\omega_t} = 10^{-3}\transp{S}S$ where the entries of $S$ sampled from a standard uniform distribution each trial. The correct iLQG solution is a locally optimal solution to the problem due to the separation principle. However, when modeling error is introduced, the iLQG solution's performance degrades rapidly. Meanwhile, \emph{the TRV-based control policy is a reliable (i.e. lower variance) and performant control strategy despite this modeling error.} In addition, the solution found by our algorithm satisfies $\mathrm{rank}(C_t) = 1$ for all $t$. Therefore, \emph{the online estimator needs to only track a single TRV} corresponding to this subspace.

\begin{figure}[t]
\begin{center}
\includegraphics[width=0.8\columnwidth]{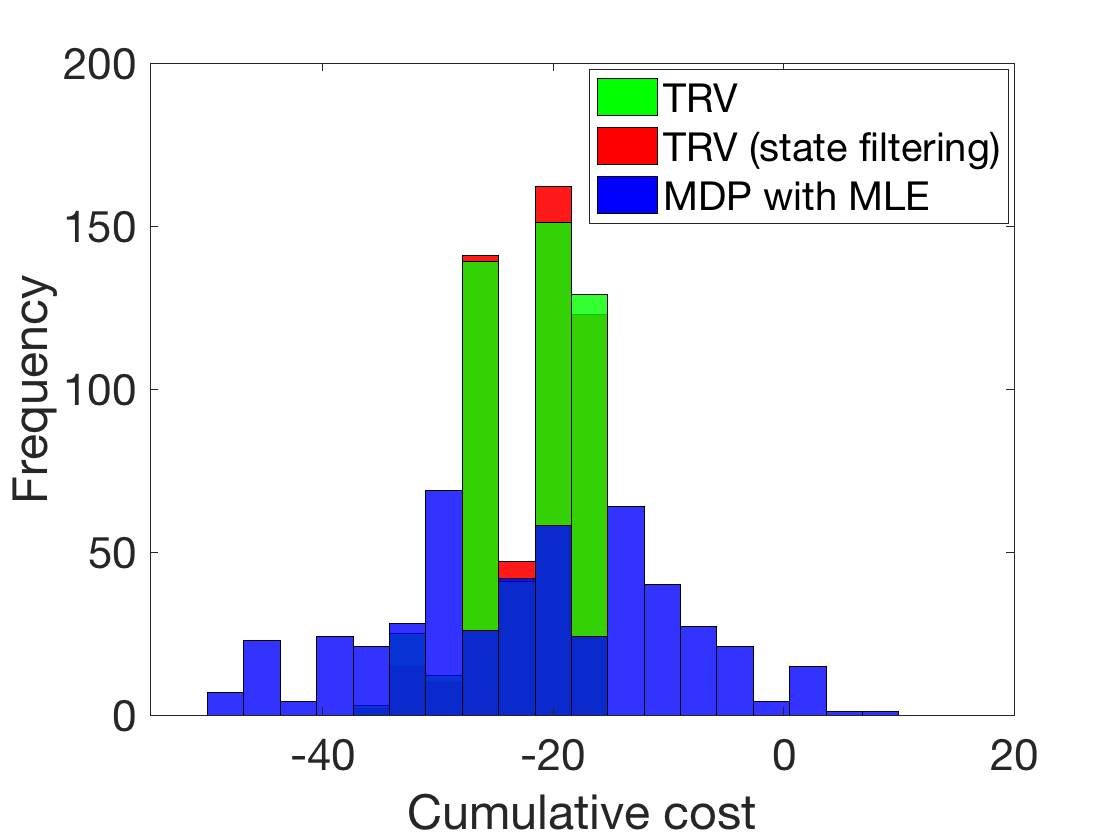}
\end{center} 
\vspace{-14pt}
\caption{This figure summarizes the outcome of 500 simulations of the Lava Problem with different control strategies. Each controller used a Bayesian filter to track the current belief distribution. The exact MDP solution (blue) applied the input corresponding to its MLE state. The TRV solutions sampled from the conditional distribution corresponding to their stochastic control policies given the MLE estimates of the state (red) or TRV (green). \label{fig:lava_results}}
\vspace{-20pt}
\end{figure}

\begin{figure}[t]
\begin{center}
\includegraphics[width=0.8\columnwidth]{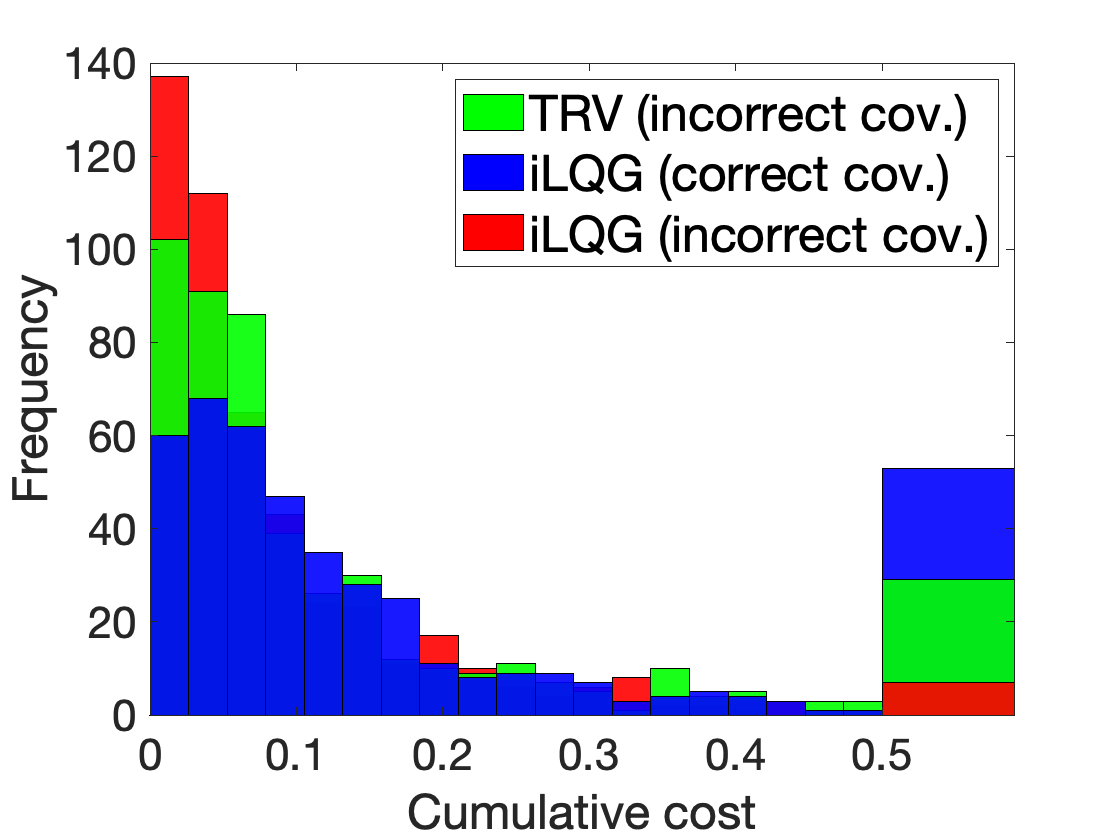}
\end{center} 
\vspace{-14pt}
\caption{This figure summarizes the outcome of 500 simulations of the SLIP Problem with different control strategies. Measurement covariance matrices were randomly sampled. For the iLQG control policy, a Kalman filter tracked the current state estimate, and the corresponding control was applied. For the TRV policy, a Kalman filter was maintained on the TRVs, and the control input corresponding to the MLE TRV was applied. \label{fig:slip_results}}
\vspace{-18pt}
\end{figure}

%\begin{figure*}[!b]
%\vspace{2mm}
%\centering
%\begin{tabular}{@{}c@{\hspace{0.5mm}}@{}c@{\hspace{0.5mm}}@{}c@{\hspace{0.5mm}}@{}c@{}}
%\includegraphics[width=0.248\textwidth]{figures/lava_results.pdf} &
%\includegraphics[width=0.248\textwidth]{figures/lava_results.pdf} &
%\includegraphics[width=0.248\textwidth]{figures/lava_results.pdf} &
%\includegraphics[width=0.248\textwidth]{figures/lava_results.pdf} 
%\\[-1.5mm] \scalebox{0.7}{(a) Lava Problem} & \scalebox{0.7}{(b) Gaze Heuristic Problem} & \scalebox{0.7}{(c) SLIP Model LQG} & \scalebox{0.7}{(d) SLIP Model ILQR}
%\end{tabular}
%\vspace{-2.5mm}
%\caption{Lorem Ipsum...}
%\label{fig:results}
%\end{figure*}

\section{Conclusion}

We presented an algorithmic approach for task-driven estimation and control for robotic systems. In contrast to prevalent approaches for control that rely on accurate full-state estimation, our approach synthesizes a set of task-relevant variables (TRVs) that are sufficient for achieving the task. Our key insight is to pose the search for TRVs as an information bottleneck optimization problem. We solve this problem offline to identify TRVs and a control policy based on the TRVs. Online, we only estimate the TRVs to apply the policy. Our theoretical results suggest that this approach affords robustness to \emph{unmodeled} measurement uncertainty. This is validated by thorough simulations, including a SLIP model running to a target location. Our simulations also demonstrate that our approach finds highly compressed TRVs (e.g. a \emph{one-dimensional} TRV for the SLIP model). 

{\bf Challenges and Future Work.} On the algorithmic front, we plan to develop approaches that directly minimize the RHS of \eqref{eq:cost bound} instead of a linear approximation of it. We expect that this may lead to improved robustness (as suggested by Theorem \ref{thm:robustness}). On the practical front, we plan to implement our approach on a hardware platform that mimics the gaze heuristic and other examples including navigation problems (where the full state includes a description of the environment, e.g. in terms of an occupancy map). Perhaps the most exciting direction is to explore \emph{active} versions of our approach where the control policy minimizes task-relevant uncertainty, in contrast to current approaches (e.g., belief space planning) that minimize full-state uncertainty. 

We believe that the approach presented in this paper along with the indicated future directions represent an important step towards developing a principled, general framework for task-driven estimation and control.
\iftoggle{extended}{}{\pagebreak}

\iftoggle{extended}{
\appendix
This appendix provides proofs for the main theorems stated in this article.

\subsection{Proof of \cref{thm:boltzmann}}
\label{apdx:boltzmann}

The first step of this derivation is to determine the structure of the value function, $\nu_t(x_t)$. The functional derivative of $\mathcal{L}$ with respect to $p_t(x_t)$ is
\begin{flalign}
\frac{\delta \mathcal{L}}{\delta p_t(x_t)} &= \underbrace{\sum_{\tilde{x}, u} c_t(x_t, u) \pi_t(u | \tilde{x}_t) q_t(\tilde{x} | x_t)}_{\mathbb{E}(c_t | x_t)}\nonumber\\
	&+ \frac{1}{\beta}\underbrace{\sum_{\tilde{x}}q_t(\tilde{x}|x_t)\log\left(\frac{q_t(\tilde{x}|x_t)}{q_t(\tilde{x})}\right)}_{\KL\left(q(\tilde{x}|x_t)||q_t(\tilde{x})\right)}
	- \nu_t(x_t)\\
	&+ \underbrace{\sum_{u, \tilde{x}, x'}\nu_{t + 1}(x')p_t(x'|x_t, u)\pi_t(u|\tilde{x})q_t(\tilde{x}|x_t)}_{\mathbb{E}(\nu_{t + 1}|x_t)}.\nonumber
\end{flalign}
Invoking the FONC yields
\begin{flalign}
	\nu_t(x_t) = \mathbb{E}(c_t + \nu_{t + 1}| x_t) + \frac{1}{\beta}\KL\left(q(\tilde{x}|x_t)||q_t(\tilde{x})\right). \label{eq:value func apdx}
\end{flalign}

Next, we repeat the process for the decision variable $q_t(\tilde{x}_t|x_t)$. The functional derivative of $\mathcal{L}$ is
\begin{flalign}
\frac{\delta \mathcal{L}}{\delta q_t(\tilde{x}_t|x_t)} = &\bigg[\underbrace{\sum_u c_t(x_t, u) \pi_t(u | \tilde{x}_t)}_{\mathbb{E}(c_{t} | x_t, \tilde{x}_t)} - \frac{1}{\beta} \log\left(\frac{q_t(\tilde{x_t}|x_t)}{q(\tilde{x}_t)}\right)\nonumber\\
	&+ \underbrace{\sum_{u, x'}\nu_{t + 1}(x')p_t(x' | x_t, u)\pi_t(u | \tilde{x}_t)}_{\mathbb{E}(\nu_{t + 1} | x_t, \tilde{x}_t)}\bigg]p_t(x_t)\nonumber\\
	&+ \lambda_t(x_t),
\end{flalign}
where $\lambda_t(x_t)$ is the Lagrange multiplier corresponding to the normalization constraint of $q_t(\tilde{x}_t | x_t)$. Invoking the FONC again and solving for $q_t(\tilde{x}_t | x_t)$ yields
\begin{flalign}
	q_t(\tilde{x}_t | x_t) = \frac{q_t(\tilde{x}_t)\exp\left(-\beta\EE(\nu_{t + 1} + c_t | x_t, \tilde{x}_t)\right)}{\exp(\beta \lambda_t(x_t) / p_t(x_t))}.
\end{flalign}
Since $q_t(\tilde{x}_t | x_t)$ is a probability distribution and must be normalized, it is the case that $\exp(\beta \lambda_t(x_t) / p_t(x_t)) = Z_t(x_t)$ where
\begin{flalign}
	Z_t(x_t) &= \sum_{\tilde{x}}q_t(\tilde{x})\exp\left(-\beta[\EE(c_t + \nu_{t + 1}| x_t, \tilde{x}))]\right). \label{eq:code given state apdx}
\end{flalign}

\subsection{Proof of \cref{thm:lg sln}}
\label{apdx:lg sln}

Again, the first step of this proof is to establish the structure of the value function $\nu_t(x_t)$. By the previous theorem, the value function at the terminal time step is $\nu_T(x_T) = c_T(x_T)$. This function is the quadratic $\nu_T(x_T) = \transp{x}_T P_T x_T + \transp{b}_{T} x_T + \const$ where:
\begin{align}
	P_{T} &= Q_{T + 1},& b_{T} &= -Q_{T} g_{T}.
\end{align}
Define $G_t \coloneqq \transp{C_t} \inv{(C_t \Sigma_{x_t} C_t + \Sigma_{\eta_t}})C_t$. Through use of the fact that the KL-divergence is the quadratic
\begin{flalign}
	\KL\left(q(\tilde{x}|x_t)||q_t(\tilde{x})\right) &= \frac{1}{2}\transp{(\bar{x}_t - x_t)} G_t (\bar{x}_t - x_t) + \const,\nonumber
\end{flalign}
recursively plugging in $\nu_{t + 1}(x_{t + 1})$ into \cref{eq:value func apdx} shows that $\nu_t(x_t) = \transp{x}_t P_T x_t + \transp{b}_{T} x_t + \const$ with
\begin{align}
	P_{t + 1} &= Q_t + \beta^{-1}G_t + \transp{C}_t \transp{K}_t R_t K_t C_t\nonumber\\
	&+ \transp{(A_t + B_t K_t C_t)}P_{t + 1}(A_t + B_t K_t C_t),\\[10pt]
	b_t &= \transp{(A_t + B_t K_t C_t)} P_{t + 1} B_t K_t a_t -Q_t g_t - \inv{\beta} G_t \bar{x}_t \nonumber\\
	 &+ \transp{C}_t \transp{K}_t R_t K_t a_t + \transp{(A_t + B_t K_t C_t)}b_{t + 1}.
\end{align}

Next, the logarithm of \cref{eq:code given state apdx} is the quadratic
\begin{flalign}
\log q_t(\tilde{x}_t&|x_t) = -\frac{1}{2}\transp{\tilde{x}}_t\left(\inv{\Sigma}_{\tilde{x}_t} + \beta \transp{K}_t  (\transp{B}_t P_{t + 1} B_t + R_t) K_t\right)\tilde{x}_t \nonumber\\
	- &\big(\beta \transp{K}_t \transp{B}_t (b_{t + 1} + P_{t + 1} B_t h_t + P_{t + 1} A_t x_t)\nonumber\\
	- &\inv{\Sigma}_{\tilde{x}_t|x_t}\bar{\tilde{x}}_t + \beta \transp{K}_t R_t (h_t - w_t)\big)^{\mathrm{T}}\tilde{x}_t + \const	
\end{flalign}
Since the logarithm of $q_t(\tilde{x}_t|x_t)$ is quadratic, $q_t(\tilde{x}_t|x_t)$ is a Gaussian distribution with mean and covariance:
\begin{flalign}
	\mu_{\tilde{x}|x_t} =\ &\beta \transp{K}_t \transp{B}_t (b_{t + 1} + P_{t + 1} B_t h_t + P_{t + 1} A_t x_t)\nonumber\\
	&-\inv{\Sigma}_{\tilde{x}_t|x_t}\bar{\tilde{x}}_t + \beta \transp{K}_t R_t (h_t - w_t),\\
	\Sigma_{\tilde{x}|x_t} =\ &\inv{\Sigma}_{\tilde{x}_t} + \beta \transp{K}_t  (\transp{B}_t P_{t + 1} B_t + R_t) K_t.
\end{flalign}
Since $p_t(x_t)$ and $q_t(\tilde{x}_t | x_t)$ are related via the linear-Gaussian $\tilde{x}_t = C_t x_t + a_t + \eta_t$ where $\eta_t \sim N(0, \Sigma_{\eta_t})$, it is necessary that
\begin{flalign}
	C_t =\ &-\beta \Sigma_{\eta_t}\transp{K}_t \transp{B}_t P_{t + 1} A_t,\\[10pt]
	a_t =\ &-\Sigma_{\eta_t}\big(\beta \transp{K}_t \transp{B}_t (b_{t + 1} + P_{t + 1} B_t h_t)\nonumber\\
	 &+ \beta \transp{K}_t R_t (h_t - w_t) - \inv{\Sigma}_{\tilde{x}_t}(C_t \bar{x}_t + a_t)\big),\nonumber\\[10pt]
	\Sigma_{\eta_t} &= \inv{\left(\inv{\Sigma}_{\tilde{x}_t} + \beta \transp{K}_t  (\transp{B}_t P_{t + 1} B_t + R_t) K_t\right)}.\nonumber	
\end{flalign}

\subsection{Proof of \cref{thm:kalman}}
\label{apdx:kalman}

The distribution of the conditional random variable $x_t | \tilde{x}_t \sim N(\mu_{x_t | \tilde{x}_t}, \Sigma_{x_t | \tilde{x}})$ is the Gaussian distribution given by the minimum mean square error estimate (MMSE) estimate of $x_t$ given $\tilde{x}_t$, i.e.
\begin{flalign}
	\mu_{x_t | \tilde{x}_t} &= \bar{x}_t + \Sigma_{x_t} \transp{C}_t \inv{\Sigma}_{\tilde{x}_t}	(\tilde{x}_t - C_t \bar{x}_t - a_t),\\
	\Sigma_{x_t | \tilde{x}_t} &= \Sigma_{x_t} - \Sigma_{x_t} \transp{C}_t \inv{\Sigma}_{\tilde{x}_t}C_t \Sigma_{x_t}.\nonumber
\end{flalign}
First, we compute the measurement model relating $y_t$ and $\tilde{x}_t$. Since $y_t = D_t x_t + \omega_t$ where $\omega_t \sim N(0, \Sigma_{\omega_t})$, we have $y_t|\tilde{x}_t \sim N(\mu_{y_t | \tilde{x}_t}, \Sigma_{y_t|\tilde{x}_t})$ with
\begin{flalign}
	\mu_{y_t | \tilde{x}_t} &=  D_t \bar{x}_t + D_t \Sigma_{x_t} \transp{C}_t \inv{\Sigma}_{\tilde{x}_t}(\tilde{x}_t - \bar{\tilde{x}}_t),\\
	\Sigma_{y_t | \tilde{x}_t} &= D_t \Sigma_{y_t | \tilde{x}_t} \transp{D}_t - D_t \Sigma_{x_t} \transp{C}_t \inv{\Sigma}_{\tilde{x}_t} C_t \Sigma_{x_t} \transp{D}_t + \Sigma_{\omega_t},\nonumber
\end{flalign}
The affine change of variables
\begin{align}
\tilde{D}_t &\coloneqq D_t \Sigma_{x_t} \transp{C}_t \inv{\Sigma}_{\tilde{x}_t},& \tilde{y}_t &\coloneqq y_t - D_t \bar{x}_t + \tilde{D}_t \bar{\tilde{x}}_t,
\end{align}
defines a linear-Gaussian output in the common form:
\begin{align}
	\tilde{y}_t &= \tilde{D}_t \tilde{x}_t + \tilde{\omega}_t, & \tilde{\omega}_t &\sim N(0, \Sigma_{y_t | \tilde{x}_t}).\label{eq:trv meas}
\end{align}

Now considering the process model, the distribution $x_{t + 1} | \tilde{x}_t \sim N(\mu_{x_{t + 1} | \tilde{x}_t}, \Sigma_{x_{t + 1} | \tilde{x}})$ is given by
\begin{flalign}
	\mu_{x_{t + 1} | \tilde{x}_t} &= A_t \mu_{x_{t} | \tilde{x}_t} + B_t u_t,\\
	&= A_t \bar{x}_t + A_t \Sigma_{x_t} \transp{C}_t \inv{\Sigma}_{\tilde{x}_t}(\tilde{x}_t - \bar{\tilde{x}}_t) + B_t u_t,	\vspace{5pt}\nonumber\\[10pt]
	\Sigma_{x_{t + 1}|\tilde{x}_t,u_t} &= A_t \Sigma_{x_t | \tilde{x}_t} \transp{A}_t + \Sigma_{\epsilon_t},\nonumber
\end{flalign}
and
\begin{flalign}
	\mu_{\tilde{x}_{t + 1} | \tilde{x}_t} &= C_{t + 1} A_t \mu_{x_{t} | \tilde{x}_t} + C_{t + 1} B_t u_t,\\
	\Sigma_{\tilde{x}_{t + 1}| \tilde{x}_t,u_t} &= C_{t + 1} \Sigma_{x_{t + 1} | \tilde{x}_t} \transp{C}_{t + 1} + \Sigma_{\eta_{t + 1}}.\nonumber
\end{flalign}
This distribution can be written in the common form
\begin{flalign}
	\tilde{x}_{t + 1} &= \tilde{A}_t \tilde{x}_t + \tilde{B}_t u_t + \tilde{r}_t + \tilde{\epsilon}_t	, \label{eq:trv proc}
\end{flalign}
where
\begin{flalign}
	\tilde{A}_t &\coloneqq C_{t + 1} A_t \Sigma_{x_t} \transp{C}_t \inv{\Sigma}_{\tilde{x}_t}, & \tilde{B}_t &\coloneqq C_{t + 1} B_t,\\
	\tilde{r}_t &\coloneqq -\tilde{A}_t \bar{\tilde{x}}_t + C_{t + 1} A_t \bar{x}_t + a_{t + 1}, & \tilde{\epsilon}_t &\sim N\left(0, \Sigma_{\tilde{x}_t|\tilde{x}_t, u_t}\right),\nonumber
\end{flalign}

The equations \cref{eq:trv meas} and \cref{eq:trv proc} constitute a LG system with $\tilde{x}_0 \sim N(C_0 x_0 + d_0, C_0 \Sigma_{x_0} \transp{C}_0 + \Sigma_{\eta_0})$. Despite the statistical properties of this system's noise being linked through common parameters, e.g. $\Sigma_{\tilde{x}_t}$, they are independent random variables --- given any subset of noise variables, the distribution of the remaining noise variables does not change. Since this LG system describes the evolution of $q_t(\tilde{x}_t)$ in time, the Bayesian updates for a filter on $\tilde{x}_t$ are given by the standard Kalman filter on this induced LG system.
}{}

\bibliographystyle{IEEEtran}
\bibliography{irom}

\end{document}